\documentclass[12pt]{amsart}
\usepackage{amsmath,amssymb,amsfonts,amscd,verbatim,graphicx}

\usepackage{tikz}
\usetikzlibrary{arrows}
\usepackage{tikz-cd}
\usetikzlibrary{cd}

\usepackage{scalerel}[2016/12/29]
\newcommand\PLSub{\scaleobj{.8}{\stackrel {\Small \rm PL}{\subset}}}

\usepackage{array,multirow}
\usepackage{color}

\usepackage{hyperref}
\hypersetup{
    colorlinks = true,
    linkcolor = {black},
   citecolor  = {black},
   urlcolor = {black}
}

\usepackage[margin=1in]{geometry}
\usepackage{lmodern}
\usepackage{newtxtext}
\usepackage{newtxmath}
\usepackage{amsthm}
\usepackage{amssymb}
\usepackage{tikz}
	\usetikzlibrary{decorations.pathreplacing}
	\usetikzlibrary{patterns}

\usepackage{xcolor}

\newtheorem{lemma}{Lemma}[section]

\newtheorem{proposition}[lemma]{Proposition}

\theoremstyle{definition}

\theoremstyle{remark}
\newtheorem{remark}[lemma]{Remark}
\numberwithin{equation}{section}
\theoremstyle{plain}
\newtheorem{thm}{Theorem}

\theoremstyle{definition}

\theoremstyle{definition}

\newtheorem*{defin*}{Definition}

\theoremstyle{plain}

\begin{document}

\title{Spineless $5$-manifolds and the deformation conjecture}

\author{Michael Freedman}
\address{\hskip-\parindent
  Michael Freedman, %\newline
CMSA, Harvard University, Cambridge, MA 02138
\newline
    Department of Mathematics\\
    University of California, Santa Barbara\\
    Santa Barbara, CA 93106}
\email{mikehartleyfreedman@outlook.com}

\author{Vyacheslav Krushkal}
\address{\hskip-\parindent
  Slava Krushkal,
  %\newline 
    Department of Mathematics\\
    University of Virginia\\
    Charlottesville, VA 22904}
\email{krushkal@virginia.edu}

\author{Tye Lidman}
\address{\hskip-\parindent
  Tye Lidman,
  %\newline
  Department of Mathematics\\
  North Carolina State University\\
  Raleigh, NC 27695}
\email{tlid@math.ncsu.edu}

\begin{abstract} We construct a compact PL $5$-manifold $M$ (with boundary) which is homotopy equivalent to the wedge of eleven $2$-spheres, $\vee^{}_{1\! 1}S^2$,
which is ``spineless'', meaning $M$ is not the regular neighborhood of any $2$-complex PL embedded in $M$.
We formulate a related question about the existence of exotic smooth structures on $4$-manifolds which is of interest in relation to the deformation conjecture for $2$-complexes, also known as the generalized Andrews-Curtis conjecture. 
\end{abstract}

\maketitle

\section{Introduction}

The purpose of this paper is to give an application of the existence of exotic smooth structures on $4$-manifolds to a question about spines in classical PL topology, and to propose an approach to the deformation conjecture for $2$-complexes (or equivalently, group presentations). To state the results, we start by recalling some facts and questions about PL manifolds and $2$-complexes. 

\subsection{Spines} Since the discussion of spines mixes simplicial\footnote{There is no loss of generality in assuming all cell complexes we encounter to be simplicial, and we make this assumption in this subsection.} complexes and manifolds, the most convenient category for our manifolds is PL. Our focus will be on $5$-manifolds, a dimension where every PL manifold has a unique smoothing (since ${\rm PL}/{\rm O}$ is $6$-connected), so it is harmless for the reader to think in the smooth category.

If $M$ is a PL manifold with boundary which admits a {\em collapse} to a complex $K\subset {\rm int}(M)$ then $K$ is a {\em spine} of $M$. A ``collapse'' means a sequence of elementary collapses, see \cite{Cohen}. Equivalently, $K\subset M$ is a spine of $M$ if and only if $M$ is a regular neighborhood of $K$.
We will only be concerned with the case ${\rm dim}(K)=k$ and ${\rm dim}(M)=2k+1$. With this restriction in mind, we say a manifold with boundary $M^{2k+1}$ is {\em spineless} if and only if there is no spine $K^k\subset M^{2k+1}$. Contrariwise, we simply say $M^{2k+1}$ has a {\em spine} if there exists a $k$-dimensional spine\footnote{Note all manifolds $M^n$ with boundary have a spine of dimension $n-1$, by collapsing top simplices. Finding lower dimensional spines takes more work.} $K\subset M^{2k+1}$. 
Our main result is the following theorem.

\begin{thm} \label{spineless M}
There exists a $5$-manifold $M$,
(simple) homotopy equivalent to $\vee^{}_{1\!1}S^2$, which does not have a spine.
\end{thm}

\begin{remark}
The notion of a spine considered here is the classical one used in PL topology. A different meaning of the term ``spine'' has also been used in the literature, referring to a PL embedding $L\subset M$ which induces a homotopy equivalence,  where $L$ is a closed PL manifold. 
Examples of manifolds of even dimensions $\geq 6$ that do not admit a codimension two spine in this sense were given in \cite{CS}, and in dimension four they are due to \cite{Matsumoto, LL}.
Note that a homotopy equivalence $K^2\to M^5$ may be assumed to be a PL embedding by general position.
\end{remark}

\begin{remark}
The manifold $M$ relies on the existence of an exotic $\#_{1\!1} (S^2 \times S^2)$ with non-vanishing Seiberg-Witten invariants, established by Baykur-Hamada in \cite{BH}.  This will be the boundary of $M$.  Since they construct  infinitely many in this TOP homeomorphism class,  one can obtain infinitely many homotopy equivalent spineless manifolds which are not PL homeomorphic as their boundaries differ.  In general, our arguments show that any simply-connected 4-manifold with vanishing signature, non-vanishing Seiberg-Witten invariants, and $b^+ > 1$ leads to a spineless 5-manifold.  Hence, we could use earlier constructions of exotica, such as those in \cite{Park}, or other examples from \cite{BH}, to produce spineless 5-manifolds with larger $b_2$.  We chose to focus on the $\#_{1\!1} (S^2 \times S^2)$ from \cite{BH} for concreteness since they are currently the smallest known spin examples.       
\end{remark}

Complementary to Theorem~\ref{spineless M}, we also establish the following result, concerning spines in other dimensions, which is likely well-known.

\begin{thm} \label{thm: k not 2}
Let $M^{2k+1}$ be a PL manifold simple homotopy equivalent to a $k$-complex $K$, then, if $k\neq 2$, $M$ has $K$ as a spine.
\end{thm}

\subsection{Deformations of \texorpdfstring{$2$}{2}-complexes} \label{sec:2complexes}
Although in the previous subsection we thought of our complexes as simplicial so as to discuss PL embeddings in PL manifolds, it is sometimes more convenient in this subsection  to think of CW $2$-complexes to match group theory: generators and relations. Since we will freely allow low dimensional deformations, $2$-complexes may be thought of as group presentations: collapse a maximal tree in the $1$-skeleton to see a wedge of circles with $2$-cells attached. The Andrews-Curtis conjecture (see \cite{KK} and references therein) is the most famous open problem about group presentations. In geometric language it asks if the $2$-complex associated with a balanced presentation of the trivial group can always be {\em $3$-deformed} to the empty presentation, i.e. a point.

A {\em $3$-deformation}\footnote{$3$-deformation can be defined group theoretically. See \cite[Section 2]{KK} and references therein for a precise definition.} is, according to the usage in simple homotopy theory, a deformation between (in our case) $2$-complexes involving no expansion beyond $3$-cells. 

By the {\em deformation conjecture} we mean the generalization of the Andrews-Curtis conjecture, stating that any two simple homotopy equivalent $2$-complexes are related by a $3$-deformation. We refer the reader to \cite{HM} for a survey of the generalized Andrews-Curtis conjecture; it was termed the `deformation conjecture' in \cite{Quinn}. 
%In Section \ref{sec: discussion} 
%we pose a question about the existence of exotic smooth structures on certain bounding $4$-manifolds, which is motivated by the proof of Theorem \ref{spineless M}, and is of interest in relation to the deformation conjecture.  But first, in Sections~\ref{sec: 5manifold} and \ref{sec:k not 2} we prove Theorem~\ref{spineless M} and Theorem~\ref{thm: k not 2} respectively.

\subsection{Exotic smooth structures on $4$-manifolds and the deformation conjecture} \label{sec: discussion}

The following statement follows from facts about deformations of spines of $5$-manifolds, see Section \ref{sec: Deformations of spines}.

\begin{proposition} \label{prop: 5 manifolds}
If the deformation
conjecture is true, then any two $5$-manifolds $M_1, M_2$ which are simple-homotopy equivalent and contain spines are PL isomorphic. In particular, their $4$-manifold boundaries $\partial M_1, \partial M_2$ are diffeomorphic.
\end{proposition}

The proof of Theorem \ref{spineless M}, given in Section \ref{sec: 5manifold}, relies on the fact that if a $5$-manifold $M$ has a handle decomposition with all handles of indices $\leq 2$ and $b^+(\partial M)>1$ then $\partial M$ contains an embedded homologically essential square zero $2$-sphere and thus its Seiberg-Witten invariants vanish. 
On the other hand, there are $4$-manifold invariants, cf. \cite{LLP}, which can distinguish homeomorphic smooth 4-manifolds containing square zero $2$-spheres.  However, we are not aware of instances of this where the exotic pairs bound any 5-manifolds.  
With this in mind, we formulate the following question, which in fact was the original  motivation for this paper.

{\bf Question}. {\em
Do there exist exotic pairs of $4$-manifolds $N_1, N_2$ such that $N_i=\partial M_i$, where $M_1, M_2$ are simple-homotopy equivalent $5$-manifolds admitting handle structures with all handles of indices $\leq 2$? 
}

By the discussion in Section \ref{sec: Deformations of spines}, the condition above on the handle decompositions is equivalent to the requirement that $M_1, M_2$ have $2$-spines. By Proposition \ref{prop: 5 manifolds}, the affirmative answer to the question would give a counterexample to the deformation conjecture.

\section{On the existence of spines in  \texorpdfstring{$5$}{5}-manifolds} \label{sec: 5manifold}

We start this section by recording some statements about spines of $5$-manifolds and their $3$-deformations. The proof of Theorem \ref{spineless M} is given in Section \ref{sec: the manifold}.

\subsection{Deformations of spines} \label{sec: Deformations of spines}
The following two facts are used in the proofs of Theorem \ref{spineless M} and of Proposition \ref{prop: 5 manifolds} respectively.

\noindent
{\bf Facts}:\\
(1)  \parbox[t]{14.5cm}{ A 5-manifold
$M$ admits a spine if and only if it it has a handle decomposition with only 0-, 1-, and 2-handles}

(2) \parbox[t]{14.5cm}{  If $K_1$ $3$-deforms to $K_2$ and $K_1\subset M$ is a spine, then there exists a PL embedding of $K_2$ in $M$ so that $K_2$ is also a spine of $M$.}

For $2$-complexes in a $5$-manifold there are no flatness issues so $0$-, $1$- and $2$-cells may be thickened to $5$-dimensional $0$-, $1$- and $2$-handles. Indeed, the local link models to flatten are $S^1 \, \PLSub \, S^4$ (for interior points) and $([0,1], \{ 0, 1 \}) \, \PLSub \, (B^4, \partial)$ (for boundary points). Both are PL unknotted.
The second fact is derived using the observation that an expansion followed by a collapse of a $3$-cell can be seen as sliding one $2$-cell over other $2$-cells. 
Sliding a $2$-handle over another $2$-handle requires the ability to take parallel copies of $2$-handles and to connect these by disjoint bands. 
Once an isotopy is built, PL ambient isotopy \cite{Hudson} applies slide-by-slide to build an ambient isotopy taking $N(K_1)$ to $N(K_2)$.

Proposition \ref{prop: 5 manifolds} follows from (2).
Indeed, the deformation conjecture, if true, would imply that 2-spines of $M_1, M_2$ are related by a 3-deformation, which translates to a sequence of handle slides, showing that the boundaries  $N_1, N_2$ are diffeomorphic.

\subsection{The \texorpdfstring{$5$}{5}-manifold \texorpdfstring{$M$}{M}}
\label{sec: the manifold}

It goes back to Wall \cite{Wall} that any two simply connected $4$-manifolds with the same homotopy type are (smoothly) $h$-cobordant. We will use this fact, but in the PL category. Let BH denote one of the exotic $\#_{1\!1} (S^2 \times S^2)$ constructed by Baykur-Hamada in \cite{BH} mentioned in the introduction.  Let $\#^{}_{ 1\!1} (S^2\times S^2)$ denote the same manifold but with the standard smooth structure. Further, let $(W; \, \#^{}_{ 1\!1} (S^2\times S^2), \, {\rm BH})$ denote the $h$-cobordism between the two (it happens to be unique). Now define 
\[ M^5:= (\natural^{}_{1\! 1} S^2\times D^3 )\cup^{}_{{\rm id\; on} \; \partial} W.\]
Clearly, $M$ is homotopy equivalent to $\vee^{}_{1\! 1} S^2\simeq \natural^{}_{1\!1} (S^2\times D^3)$.

\begin{thm}\label{thm:spineless}
The manifold $M$ cannot be built using only 0-, 1-, and 2-handles.
\end{thm}

By the discussion in Section \ref{sec: Deformations of spines}, Theorem \ref{thm:spineless} implies Theorem \ref{spineless M} in the introduction.

To prove Theorem \ref{thm:spineless}, since the BH manifolds are simply-connected with $b^+ = 11$ and have non-vanishing Seiberg-Witten invariants, it suffices to establish the following proposition.  
\begin{proposition}\label{prop:vanishing}
Let $X$ be a compact, connected, oriented smooth 5-manifold with $b^+(\partial X) > 1$.  If $X$ can be built from only 0-,1-, and 2-handles, then the Seiberg-Witten invariants of $\partial X$ vanish.  
\end{proposition}

First, we need a standard lemma.  Recall that given an embedded loop $\gamma$ in an oriented 4-manifold $N$, we can perform surgery by removing an $S^1 \times D^3$ and regluing by a $D^2 \times S^2$.  Denote the result by $N_\gamma$.  There are two framing choices here, but the arguments are unaffected by this choice, so we suppress this from the notation.  Note that $H_1(N_\gamma)$ is isomorphic to $H_1(N) / \langle \gamma \rangle$.  
\begin{lemma}\label{lem:surgery}
Let $N$ be a closed, oriented 4-manifold and $\gamma$ an embedded loop.  Let $N_\gamma$ denote the result of surgery on $\gamma$ with some choice of framing.  If $\gamma$ is non-trivial in $H_1(N;\mathbb{Q})$, then $b_2(N_\gamma) = b_2(N)$.  If $\gamma$ is rationally nullhomologous, then $b_2(N_\gamma) = b_2(N) + 2$.  Finally, in the rationally nullhomologous case, $N_\gamma$ has an embedded square zero sphere which is non-zero in $H_2(N_\gamma;\mathbb{Q})$.  
\end{lemma}
Note that there is a 5-dimensional 2-handle cobordism $(Z; N, N_\gamma)$. The claimed essential square zero $2$-sphere is the belt sphere of the cobordism $Z$.
The proof of the lemma is an exercise in homology calculations.

\begin{proof}[Proof of Proposition~\ref{prop:vanishing}]
The key input is that Fintushel-Stern proved a 4-manifold with $b^+ > 1$ and a rationally essential embedded square zero sphere has vanishing Seiberg-Witten invariants \cite[Lemma 5.1]{FSImmersed}.  We will establish the existence of such a sphere.  Suppose $X$ is built from one 0-handle, $g$ 1-handles, and $n$ 2-handles.  Then $\partial X$ is described by taking $\#_g (S^1 \times S^3)$ and surgering $n$ loops, $\gamma_1,\ldots, \gamma_n$, which are the attaching circles for the 2-handles.  If $V$ is the result of surgering $\gamma_1,\ldots, \gamma_k$, for some $k$, then $H_1(V;\mathbb{Q})$ is the quotient of $H_1(\#_g S^1 \times S^3; \mathbb{Q})$ by the subspace spanned by $\gamma_1,\ldots, \gamma_k$.  After re-ordering, there is a $k$ such that $\gamma_1,\ldots, \gamma_k$ are linearly independent in $H_1(\#_g S^1 \times S^3;\mathbb{Q})$ and their span agrees with that of $\gamma_1,\ldots,\gamma_n$.  After surgering $\gamma_1,\ldots,\gamma_k$, the images of $\gamma_{k+1},\ldots, \gamma_n$ are all rationally nullhomologous.  Lemma~\ref{lem:surgery} implies that surgery on the image curves, i.e. $\partial X$, has $b_2 = 2(n - k)$.  Because $b^+(\partial X) > 1$, it follows that $\gamma_{k+1}$ exists. 
The same lemma now gives that $\partial X$ contains an embedded square zero sphere which is rationally essential, contradicting the result of Fintushel-Stern.  
\end{proof}

\section{Spines in other dimensions: proof of Theorem \ref{thm: k not 2}}\label{sec:k not 2}

\begin{proof} For $k=1$, $K$ is a graph, so ${\pi}^{}_1(K)$ is a free group. Repeated applications of Dehn's lemma/loop theorem (using the fact that any map from ${\pi}^{}_1(\partial M)$ to a free group has kernel) shows that $M^3$ compresses to a (fake) $3$-cell. It is known that $M$ must be a handlebody, which evidently has $K$ as a spine. The difficult detail that $M$ cannot contain a fake $3$-cell and is thus a standard handlebody is due to Pereleman \cite{P1, P2}.

For $k\geq 3$, we rely on the $s$-cobordism theorem. By general position assume the simple homotopy equivalence $K\to M$ is an inclusion. Let $N:=N(K)\subset {\rm int}(M)$ be the regular neighborhood and $C:=\overline{M\smallsetminus N(K)}$ be the closed complement. By the Mayer-Vietoris sequence for $M=N\cup C$ and the fact that $N\hookrightarrow M$ is a homology isomorphism, conclude that $C$ is a homology product. In the case ${\pi}^{}_1(M)\neq \{ 1\}$, make this conclusion with ${\mathbb Z}[{\pi}^{}_1(M)]$ coefficients. 

Crucially, when $k\geq 2$ the codimension of $K$ in $M$ is $\geq 3$, allowing us to show that $C$ is an $h$-cobordism.
For notation $\partial^{}_0 C :=\partial N(K)$ and $\partial^{}_1 C :=\partial M$. ${\pi}^{}_1(\partial^{}_1 C)\stackrel{{\rm inc}_{\#}}{\longrightarrow}{\pi}^{}_1(M)$ must be onto, for if not there will be kernel in the map
\[
H_0(\partial M; {\mathbb Z}[\pi^{}_1 M])\to H_0(M; {\mathbb Z}[\pi^{}_1 M]).
\]
Furthermore, 
${\pi}^{}_1(\partial^{}_0 C)\stackrel{{\rm inc}_{\#}}{\longrightarrow} {\pi}^{}_1(N)\cong {\pi}^{}_1(M)$ is an injection, since any null-homotopy $h\colon (D^2, \partial)\to (N, \partial^{}_0 C)$ will be disjoint from $K$ by general position and then can be pushed back into $\partial^{}_0 C$ using the mapping cylinder structure on $(N(K), K)$.

But since $M$ collapses to $K$, it also collapses to $N(K)$. During the collapse the fundamental group of the frontier stays constant so ${\pi}^{}_1(\partial^{}_0 C)$ and ${\pi}^{}_1(\partial^{}_1 C)$ have identical images in ${\pi}^{}_1(M)$. It follows that all the inclusions below induce isomorphisms on ${\pi}^{}_1$
\[ 
\centering
\begin{tikzcd}
\pi^{}_1(\partial^{}_0 C) \arrow[dr, "\cong" {xshift=-1.5ex,yshift=.3ex}] & & &  & \\[-15pt]
& \pi^{}_1(C) \arrow[r, "\cong"] & \pi^{}_1(M) & \pi^{}_1(N) \arrow[swap]{l}{ \cong} & \pi^{}_1(K),\arrow[swap]{l}{ \cong}
\\[-15pt]
\pi^{}_1(\partial^{}_1 C) \arrow[ur, "\cong"{xshift=1ex,yshift=-2.5ex}] & & &  &
\end{tikzcd}
\]
making $C$ an $h$-cobordism. Finally, it follows from the additivity of the Whitehead torsion that
\[
0=\tau(M,K)=\tau(N,K)+\tau(M,N) 
\]
showing $\tau(M,N)=\tau(C,\partial^{}_0 C)=0$.
Thus $(C; C_0, C_1)$ is actually an $s$-cobordism.

When $k=2$ we are in too low a dimension, $2k+1=5$, to apply the PL $s$-cobordism theorem. However for $k\geq 3$ we conclude that $C$ is a PL product $\partial^{}_0 C\times [0,1]\cong C$, implying that $K\hookrightarrow M$ is a spine.
\end{proof}

The proof makes clear that the question of $2$-spines for $5$-manifolds is in the realm of low dimensional topology. If we may digress to philosophy for a moment, bounded $5$-manifolds are inherently ``low dimensional''. Here are two examples: the failure of a smooth  or PL $5$-dimensional $h$-cobordism theorem underlies the richness of smooth $4$-manifolds. Also, the existence of topological handlebody structures on bounded $5$-manifolds was only established in \cite{FQ} using the disk embedding theorem. But the low dimensional character of bounded $5$-manifolds is often overlooked: Kirby's problem list \cite{Kirby} references ``spine'' 30 times but always in relation to 3 or 4 dimensional manifolds.

{\em Acknowledgements.} We would like to thank İnanç Baykur, Mikhail Khovanov, Peter Kronheimer, and Lisa Piccirillo for helpful discussions.

VK was supported in part by NSF Grant DMS-2105467.  TL was supported in part by NSF Grant DMS-2105469.

%\vfill

\begin{thebibliography}{alpha}



\bibitem[BH23]{BH}
R.I. Baykur, N. Hamada, {\em Exotic $4$-manifolds with signature zero}, \href{https://arxiv.org/abs/2305.10908}{arXiv:2305.10908} 

\bibitem[CS77]{CS} S.E. Cappell and J. Shaneson,
{\em Totally spineless manifolds}, Illinois J. Math. 21 (1977), 231-239.

\bibitem[Co73]{Cohen} M.M. Cohen, 
A course in simple-homotopy theory.
Grad. Texts in Math., Vol. 10
Springer-Verlag, New York-Berlin, 1973.

\bibitem[FS95]{FSImmersed} R. Fintushel and R.J. Stern, 
{Immersed spheres in {$4$}-manifolds and the immersed {T}hom conjecture}, 
Turkish J. Math. (2) 19 (1995), 145--157. 

\bibitem[FS98]{FS} R. Fintushel and R.J. Stern, {\em Knots, links, and  4-manifolds},
Invent. Math. 134 (1998),  363-400.

\bibitem[FQ90]{FQ} M. Freedman and F. Quinn, {The Topology of
$4$-Manifolds}, Princeton Math. Series {\bf 39}, Princeton, NJ, (1990).

\bibitem[HM93]{HM} C. Hog-Angeloni and W. Metzler, {\em The Andrews-Curtis conjecture and its generalizations}, London
Math. Soc. Lecture Note Ser., 197 Cambridge University Press, Cambridge, 1993, 365-380.

\bibitem[Hu66]{Hudson} J.F.P. Hudson, 
{\em Extending piecewise-linear isotopies}, Proc. London Math. Soc. (3) 16 (1966), 651-668.

\bibitem [Ki97]{Kirby} Rob Kirby (ed.), {\em Problems in low-dimensional topology}, AMS/IP Stud. Adv. Math., vol. 2, Amer. Math. Soc., Providence, RI, 1997, pp. 35-473.

\bibitem[KKN23]{KK} M. Khovanov, V. Krushkal and  J. Nicholson, {\em On the universal pairing for 2-complexes}, \href{https://arxiv.org/abs/2312.07429}{arXiv:2312.07429}

\bibitem[LL19]{LL} A. Levine and T. Lidman, {\em 
Simply connected, spineless 4-manifolds}, Forum Math. Sigma 7 (2019), Paper No. e14, 11 pp.

\bibitem[LLP23]{LLP} A. Levine, T. Lidman and L. Piccirillo, {\em New constructions and invariants of closed exotic 4-manifolds}, 
arXiv:2307.08130.


\bibitem[Ma75]{Matsumoto} Y. Matsumoto,
{\em A 4-manifold which admits no spine}, Bull. Amer. Math. Soc. 81 (1975), 467-470.

\bibitem[Pa02]{Park} J. Park, {\em The geography of Spin symplectic 4-manifolds}, Math. Z. 240 (2002), 405-421.

\bibitem[Pe02]{P1}
G. Perelman, {\em The Entropy Formula for the Ricci Flow and Its Geometric Application},  \href{http://arxiv.org/abs/math.DG/0211159}{arXiv:math/0211159}

\bibitem[Pe03]{P2}
G. Perelman, {\em Ricci Flow with Surgery on Three-Manifolds}, \href{http://arxiv.org/abs/math.DG/0303109}{arXiv:math/0303109}

\bibitem[Q85]{Quinn} F. Quinn,  {\em Handlebodies and  2-complexes}, 
Lecture Notes in Math., 1167
Springer-Verlag, Berlin, 1985, 245-259.

\bibitem[Wa64]{Wall} C.T.C. Wall, 
{\em On simply-connected $4$-manifolds}, J. London Math. Soc. 39 (1964), 141-149.


\end{thebibliography}
\end{document}